\documentclass[11pt,a4paper,reqno]{amsart}

\usepackage[english]{babel}
\usepackage{enumitem}
\usepackage{color}
\usepackage{amsmath,amsfonts,amssymb,amsopn,amscd,amsthm}

    \newtheorem{definition}{Definition}[section]
    
    \newtheorem{theorem}[definition]{Theorem}
    \newtheorem{proposition}[definition]{Proposition}
    \newtheorem{corollary}[definition]{Corollary}
    \theoremstyle{remark}
    
    \newtheorem{remark}[definition]{Remark}

\title{Asymptotic Results in Solvable Two-charge Models}
\author[M.~Dal Borgo]{ Martina Dal Borgo} 
\address[M.~Dal Borgo]{\,Institut f\"ur Mathematik, Universit\"at Z\"urich --- Winterthurerstrasse 190, 8057 Z\"urich, Switzerland}
\email{Martina.Dalborgo@math.uzh.ch}
\author[E.~Hovhannisyan]{Emma Hovhannisyan} 
\address[E.~Hovhannisyan]{\,Institut f\"ur Mathematik, Universit\"at Z\"urich --- Winterthurerstrasse 190, 8057 Z\"urich, Switzerland}
\email{Emma.Hovhannisyan@math.uzh.ch}
 \author[A.~Rouault]{Alain Rouault} 
\address[A.~Rouault]{LMV B\^atiment Fermat, Universit\'e Versailles-Saint-Quentin, France}
\email{Alain.Rouault@math.uvsq.fr}

\renewcommand{\d}{\delta}
\newcommand{\half}{\frac{1}{2}}
\renewcommand{\L}{\Lambda}
\newcommand{\g}{\gamma}
\newcommand{\G}{\Gamma}
\newcommand{\N}{\mathbb{N}}
\newcommand{\R}{\mathbb{R}}
\newcommand{\C}{\mathbb{C}}
\renewcommand{\P}{\mathbb{P}}
\newcommand{\E}{\mathbb{E}}
\newcommand{\Var}{\mathrm{Var}}
\renewcommand{\l}{\lambda}

\begin{document}
 \maketitle \vspace{-5mm}

\begin{abstract}
In \cite{RSX13}, a solvable two charge ensemble of interacting charged particles on the real line in the presence of the harmonic oscillator potential is introduced. It can be seen as a special form of a grand canonical ensemble with the total charge being fixed and unit charge particles being random. Moreover, it serves as an interpolation between the Gaussian orthogonal and the Gaussian symplectic ensembles and maintains the Pfaffian structure of the eigenvalues.  A similar solvable ensemble of charged particles on the unit circle was studied in \cite{SS14, F10}. \\
In this paper we explore the sharp asymptotic behavior of the number of unit charge particles on the line and on the circle,  as the total charge goes to infinity. We establish extended central limit theorems, Berry-Esseen estimates and precise moderate deviations using the machinery of the mod-Gaussian convergence developed in ~\cite{ FMN16, KN12, JKN11, DKN15, FMN17}. Also a large deviation principle is derived using the G\"artner-Ellis theorem. \\
\end{abstract}
33C45, 60G55, 60F05, 60F10, 62H10.

\bigskip
% --------------------------------------------------------------------
% Introduction
% --------------------------------------------------------------------
\section{Introduction}
\subsection{A general two-charge ensemble}
We first present a general model of charged particles with charge ratio $1:2$ interacting via a logarithmic potential.  Let $L, M, N$ be non-negative integers such that $L+2M=N.$ The two-component log-gas to be considered consists of $L$ particles with unit charge and $M$ particles with charge two,  located either on the line or on the unit circle and  interacting via the
logarithmic potential
\begin{align}\label{Energy}
U_{L,M}(\xi, \zeta) :=  -\sum_{1\leq i<j \leq L} \log |\xi_i - \xi_j| - 4 \sum_{1\leq i<j \leq M} \log |\zeta_i - \zeta_j| - 2 \sum_{i=1}^L \sum_{j=1}^M \log | \xi_i-\zeta_j |,
\end{align}
 where  the charge one particles are located at $\xi = \left(\xi_1, \dots, \xi_L\right)$ and charge two  particles at $\zeta = \left(\zeta_1, \dots, \zeta_M\right),$ respectively. \\ 
The system may be in the presence of an external field, and the interaction energy between the charges and the field is
$\sum_{i=1}^L V(\xi_i) + 2 \sum_{k=1}^M V(\zeta_k)$ for some potential $V$.
The total energy is then
\begin{align*}
E_{L,M} (\xi, \zeta) = U_{L,M}(\xi, \zeta) +\sum_{i=1}^L V(\xi_i) + 2 \sum_{k=1}^M V(\zeta_k).
\end{align*}
The partition function $Z_{L,M}$ of the system is 
\begin{align*}
Z_{L,M} = \frac{1}{L! M!} \int_{E^L}  \int_{E^M} e^{-E_{L,M} (\xi, \zeta)} d\mu^L(\xi) d\mu^M(\zeta) \ , \ \ (E = \R \ \hbox{or}\ \mathbb T).
\end{align*}
We study a special form of the grand canonical ensemble, where the number of charges $L$ and $M$ are random but they sum up to the total charge $N=2n$ for some $n\in \N.$ Moreover, the probability of the system having $L$ particles with unit charge and $M$ particles with charge two is equal to
\begin{align*}
X^L \frac{Z_{L,M}}{Z_N(X)},
\end{align*} 
where 
\begin{align*}
Z_N(X) = \sum_{L+2M = N} X^L Z_{L,M} 
\end{align*}
and $X$ is a parameter, called \textit{fugacity}, which determines the probability of the system having a particular population vector $(L,M)$. \\ 
In what follows we distinguish two models of two-charge ensemble, on the real line and on the unit circle. For both models,  we study the distribution of the number of charge one particles for the proper scaling of the parameter $X$, when the total charge $N$ tends to infinity.  \\
The structure of the paper goes as follows. In Subsections \ref{introduction: line} and \ref{introduction: circle} we discuss separately two-charge models on the line and on the circle, respectively. Thereafter, in Subsection \ref{introduction: mod-Gauss} we recall the definition of mod-Gaussian convergence and the limiting theorems this convergence implies. In Section \ref{results: line},  the limiting results for the two-charge models on the line are stated for two different scaling of the fugacity parameter. In Section \ref{results: circle}, we introduce the limiting theorems for the two charge ensemble on the circle.
\subsubsection{A two-charge ensemble on the real line} \label{introduction: line}
 The charged particle model on the line, was introduced in \cite{RSX13}. A generalization of these models to more than two types of charges recently has been studied in \cite{MPT15}. The motivation to study the two charged particle model on the line comes from its analogy with the real Ginibre ensemble (instead of complex conjugate pairs of eigenvalues, we have particles of charge $2$ on the real line). For more details on the complex and real eigenvalues of the real Ginibre ensembles and the existence of Pfaffian structure on the particles in these ensembles, see \cite{G65, FN07, BS09}. \\
 We study the two charge models on the line in the presence of the harmonic oscillator potential $V(x) = \frac{x^2}{2}.$ As a result, the case $X = 0$ gives exactly the Gaussian orthogonal ensemble ($\beta = 1$).  In contrast, as $X \to \infty,$ the above model corresponds to the Gaussian symplectic ensemble ($\beta = 4$). Therefore, the two-charge ensemble on the line can be interpreted as an interpolation between these two point processes.\\
We denote by $U_n$ the number of particles with unit charge  with the parameter $X.$ The aim is to investigate the limiting behavior of the sequence $\left(U_n \right)_{n \in \N},$ depending on the fugacity parameter $X.$
Typically, we will focus on two cases $X=1$ and $X=\sqrt{2n\gamma},$ for some $\gamma>0.$ The case $X = 1$ is the analog of the real Ginibre ensemble, i.e. the charge 1 particles play the role of real eigenvalues in real Ginibre and the charge 2 particles the role of pairs of complex conjugate eigenvalues.\\
For the case $X=\sqrt{2n\gamma},$ 
Rider, Sinclair and Xu, using Laplace method,  provide asymptotic formulas for the quantities $\E[U_n], \Var(U_n)$ (\cite{RSX13}). From our side, we will show large deviation principle (LDP) for the sequence $(U_n)_{n \in \N}$ as well as mod-Gaussian convergence for the normalized sequence $\left(\frac{U_n-\E[U_n]}{n^{1/3}}\right)_{n \in \N}.$ Note that the central limit theorem (CLT) will be one of the consequences of the mod-Gaussian convergence mentioned above. \\
The case $X=1$ has been studied in details in \cite{RSX13} and the limiting results such as CLT and LDP have been derived.  We extend their results by proving mod-Gaussian convergence for the sequence $\left(\frac{U_n - \E[U_n]}{n^{1/6}}\right)$ and deducing all its consequences such as precise moderate deviations and extended central limit theorem.  
In order to derive these limiting results, we will use the following characterization of the partition function, shown in Theorem 2.1 in \cite{RSX13} and following from the solvability of the ensemble 
\begin{align*}
Z_N(X)=\prod_{j=0}^{n-1} r_j(X),
\end{align*}
where 
\begin{align*}
r_j(X)=4(j+1)  \G\left( j+\half\right)  
\frac{L_{j+1}(-X^2)}{L_j(-X^2)},
\end{align*}
with $L_j(x) := L_j^{-1/2}(x)$ being the generalized $j$th Laguerre polynomial with  parameter $\alpha = -\half,$ defined as  
\[L_j^\alpha(x) = e^x \frac{x^{-\alpha}}{j!} \frac{d^j}{dx^j} (e^{-x}x^{j+\alpha})\,,\]
and orthogonal with respect to the measure on $(0, \infty)$ with density $x^\alpha e^{-x}$.

Thus, the probability generating function of $U_n$ can be written 
\begin{align}\label{2cl_prob_gen}
P_n(t) = \frac{Z_N(tX)}{Z_N(X)} = \frac{L_n(-t^2X^2)}{L_n(-X^2)}\,,
\end{align}
and then the moment generating function of $U_n$ is given by
\begin{align}\label{2cl_mom_gen}
\E\left[e^{zU_n}\right] = \frac{Z_N(e^zX)}{Z_N(X)} = \frac{L_n(-X^2 e^{2z})}{L_n(-X^2)}\,.
\end{align}

\subsubsection{A two-charge ensemble on the circle}\label{introduction: circle}
This model represents the circular version of the above model  and has been introduced and studied in \cite{F84B, F84A, F10, SS14}. We  have again $L$ particles with unit charge and $M$  particles with charge two that are located on the unit circle $\mathbb{T}$ and interact logarithmically in the absence of an external field. This ensemble forms a Pfaffian point process that interpolates between the circular and symplectic ensembles. \\
 We denote by $V_n$ the number of particles with unit charge when
 $X=n\rho,$  with $\rho>0.$ Following Section 6.7.1 of \cite{F10}, the probability generating function of $V_n$ is 
\begin{align}\label{2cc_prob_gen}
 P_n(t) = \E\left[ t^{V_n}\right] =\prod_{k=1}^n \frac{ \left(2n\rho t\right)^2 + (2k-1)^2  }{  \left(2n\rho  \right)^2 + (2k-1)^2 }.
\end{align}
Taking into account this polynomial structure, Forrester used an argument of Bender \cite{Be73} to prove the local limit theorem.  
In \cite{SS14}, the central limit theorem
has been derived for the sequence $(V_n)_{n \in \N},$ by a direct computation using 
the moment generating function, 
 \begin{align}\label{2cc_mom_gen}
 \E \left[ e^{zV_n}\right] = \prod_{k=1}^n \frac{ \left(2n\rho e^{z} \right)^2 + (2k-1)^2  }{  \left(2n\rho  \right)^2 + (2k-1)^2 }\,.
\end{align}
We will deduce the LDP for the sequence $(V_n)_{n \in \N},$ and the mod-Gaussian convergence with its consequences for the sequence $\left( \frac{V_n-\E\left[V_n\right]}{n^{1/3}}\right)_{n \in \N}.$  

% --------------------------------------------------------------------
% Mod-Gaussian convergence and limiting results
% --------------------------------------------------------------------

\subsection{The framework of mod-Gaussian convergence}\label{introduction: mod-Gauss}
The framework of mod-Gaussian convergence for a sequence of random variables has been developed by Delbaen, F\'eray, Jacod, Kowalski, M\'eliot and Nikeghbali  in~\cite{FMN16, KN12, JKN11, DKN15, FMN17}.
For $-\infty \leq c < 0 < d\leq \infty$, set
\begin{align*}
S_{(c,d)}=\{z\in\C, c < \Re (z) < d\}.
\end{align*}

\begin{definition}  Let $\left(X_n\right)_{n\in\N}$ be a sequence of real valued random variables, and $\phi_n(z)=\E\left[e^{zX_n}\right]$ be their moment generating functions, which we assume to all exist in a strip $S_{c,d}$. 
 We assume that there exists an analytic function $\psi(z)$ not vanishing on the real part of $S_{(c,d)}$, such that locally uniformly on $S_{(c,d)}$, 
\begin{align*}
\lim_{n\to\infty}\E\left[e^{zX_n}\right]e^{-t_n\frac{z^2}{2}}:=\lim_{n\to\infty}\psi_n(z)=\psi(z),
\end{align*}
where $\left(t_n\right)_{n\in\N}$ is some sequence going to infinity. We then say that 
 $\left(X_n\right)_{n\in\N}$ converges mod-Gaussian on $S_{(c,d)},$ with parameters $t_n$ and limiting function $\psi$.
\end{definition}
To establish the mod-Gaussian convergence in two charge models, the following theorem will be useful. It is a straightforward consequence of Theorem 8.2.1 in \cite{FMN16}.

\begin{theorem}
\label{newcriterion}
Let $\left(X_n\right)_{n\in\N}$ be a sequence of bounded random variables with non-negative integer values, with mean $\mu_n,$ variance $\sigma_n^2$ and third cumulant $\kappa^{(3)}_n$. 
Suppose that the probability generating function $P_n(t)$ of $X_n$ can be factorized as
\begin{align*}
P_n(t) = \prod_{1 \leq j \leq n} P_{n,j} (t)
\end{align*}
where each $P_{n,j}$ is a polynomial with non-negative coefficients and maximal degree not depending on $n$. If
\begin{align}
\label{condcumulant}
\frac{\sigma_n^2}{n} \rightarrow  \sigma^2  > 0 \ , \ \frac{\kappa^{(3)}_n}{n} \rightarrow \kappa \,,
\end{align}
then the sequence
\begin{align*}
\tilde{X}_n : = \frac{X_n -\mu_n}{n^{1/3}}
\end{align*}
converges in the mod-Gaussian sense with limiting function $\psi(z) = e^{ \frac{\kappa z^3 }{6}}$
and parameter 
$t_n = \sigma_n^2 n^{-2/3}$.
The convergence takes place of the whole complex plane.
\end{theorem}
\begin{remark}
\label{smallo}
If $\sigma_n^2 = n\sigma^2 + o(n^{2/3})$, we can take $t_n = \sigma^2 n^{1/3}$.
\end{remark}

Although mod-Gaussian convergence entails
much more information, for the aim of this paper, it will be an instrument to deduce asymptotic results, such as a central limit theorem, precise deviations etc.
In what follows we summarize some of the limiting results that mod-Gaussian convergence implies (see \cite{FMN16}, \cite{FMN17} for the results below). In the remaining part of this subsection, $\left(X_n\right)_{n\in\N}$ is a sequence converging mod-Gaussian on $S_{(c,d)},$ with parameters $t_n$ and limiting function $\psi$. 
\begin{theorem}[Precise deviations at the scale $O(t_n)$, Theorem 4.2.1 in \cite{FMN16}]\ \label{ldp}\\
For $x\in (0,d)$,
\begin{equation*}
\P\left[X_n\ge t_nx\right]=\frac{e^{-t_n\frac{x^2}{2}}}{x\sqrt{2\pi t_n}}\psi(x)\left(1+o(1)\right),
\end{equation*}
and for $x\in(c,0)$,
\begin{equation*}
\P\left[X_n\le t_nx\right]=\frac{e^{-t_n\frac{x^2}{2}}}{|x|\sqrt{2\pi t_n}}\psi(x)\left(1+o(1)\right).
\end{equation*}
\end{theorem}
\vskip 10pt
\begin{theorem}[Central limit theorem at the scale $o(t_n)$, Theorem 4.3.1 in \cite{FMN16}]\ \\
\label{clt}
For  $y=o\left(\sqrt{t_n}\right)$,
\begin{equation*}
\P\left[\frac{X_n}{\sqrt{t_n}}\ge y\right]=\P\left[
\mathcal{N}(0,1)\ge y\right]\left(1+o(1)\right)=\frac{e^{-\frac{y^2}{2}}}{y\sqrt{2\pi}}\left(1+o(1)\right).
\end{equation*}
\end{theorem}

Note that Theorem \ref{clt} immediately implies a \emph{central limit theorem} for the rescaled sequence $\left(\frac{X_n}{\sqrt{t_n}}\right)_{n\in\N}$ (taking $y=O(1)$).\\
Under additional assumptions, it is possible to deduce the speed of convergence of the above mentioned CLT and a local limit theorem as well. 

\begin{definition}
Let $\left(X_n\right)_{n\in\N}$ be a sequence of real random variables, and $\left(t_n\right)_{n \in \N}$ a sequence growing to infinity. Consider the following assertions:
\begin{enumerate}[label=(\textbf{Z\arabic*}),ref=(Z\arabic*)]
\item \label{Z1} Fix $v, w > 0$ and $\gamma\in \R$.  There exists a zone $[-Dt_n^\g,Dt_n^\g]$, $D>0$, such that, for all
$\xi$ in this zone, if $\psi_n(i\xi) = \E\left[e^{i \xi X_n} \right] e^{\frac{t_n \xi^2}{2}}$, then 
\begin{align*}
\left|\psi_n(i\xi)-1\right| \leq K_1|\xi|^v e^{K_2 |\xi|^w} 
\end{align*}
for some positive constants $K_1$ and $K_2$, that are independent of $n$.
\item \label{Z2} One has 
\begin{equation}
w \geq 2;\qquad -\half \le\gamma\le\frac{1}{w-2};\qquad D\le \left(\frac{1}{4K_2}\right)^{\frac{1}{w-2}}. \nonumber
\end{equation}
\end{enumerate}
If Conditions \ref{Z1} and \ref{Z2} are satisfied, we say that we have a zone of control $[-Dt_n^\g,Dt_n^\g]$ with index $(v,w)$. 
\end{definition}

\begin{theorem}[Speed of convergence, Theorem 2.16 in \cite{FMN17}]\ \label{speed}\\
Let $\left(X_n\right)_{n\in\N}$ be  a sequence converging in mod-Gaussian sense with a zone of control $[-Dt_n^\g,Dt_n^\g]$ of index $(v,w)$. In addition, we assume $\g\le \frac{v-1}{2}$. Then,
\begin{equation*}
d_{Kol}\left(\frac{X_n}{\sqrt{t_n}}, \mathcal{N}(0,1)\right)\le C(D, v,K_1)\frac{1}{t_n^{\gamma+\half}},
\end{equation*}
where $d_{Kol}(\cdot, \cdot)$ is the Kolmogorov distance and 
\begin{equation}
\label{defc}
C(D,v,K_1) = \min_{\lambda >0} \left(\frac{1+\lambda}{\sqrt{2} \pi} \left(2^{v-\half} \G\left(\frac{v}{2} \right)K_1 + \frac{\pi^
{1/6}}{D} \left(4 \sqrt[3]{1+\frac{1}{\l}} + 3\sqrt[3]{3}\right) \right)\right).
\end{equation}
\end{theorem}
\begin{theorem}[Local limit theorem, see \cite{BMN17}]\ \label{llt}\\
 Let $x\in\R$ and $(a,b)$ be a fixed interval, with $a<b$. 
 Assume that conditions (Z1) and (Z2) hold. Then for every exponent $\d\in\left(0,\g+\frac{1}{2}\right)$, 
 \begin{equation*}
\lim_{n\to\infty}(t_n)^{\d}\P\left[\frac{X_n}{\sqrt{t_n}}-x\in\frac{1}{t_n^\d} (a,b)\right]=\frac{b-a}{\sqrt{2\pi}}.
\end{equation*}
In particular, assuming $\g>0$, with $\d=\half$ one obtains
 \begin{equation*}
\lim_{n\to\infty}(t_n)^{\half}\P\left[X_n-x\left(t_n\right)^\half \in (a,b)\right]=\frac{b-a}{\sqrt{2\pi}}.
\end{equation*}
\end{theorem}

\section{The limiting results  for the two-charge ensemble on the line} \label{results: line}
\subsection{The fugacity parameter increasing with $n$}
In this subsection we establish various asymptotic results in the case $X=\sqrt{2n\g}.$ The first theorem is the large deviation principle for the random sequence $\left(U_n\right)_{n \in \N}$ (as before $U_n$ is the number of particles of unit charge). We use the notation of Dembo and Zeitouni (see \cite{DZ10}). In particular, we denote the limiting cumulant generating function and its Legendre-Fenchel transform, by $\L$ and $\L^*$, respectively.
\begin{theorem}[LDP]\ \label{LDP_L_ng}\\
The sequence of random variables $(U_n /2n)_{n\geq 1}$ satisfies the LDP in $(0,1)$ at speed $2n$ with good rate function 
\begin{equation}
\label{reatef}
\L^*(x) =  x \log x+ \frac{1-x}{2}\log (1-x) + a(\gamma) x + b(\gamma)\, 
\end{equation}
where $a(\gamma) = -\half (1 + \log(2\gamma))$ and $b(\gamma)$ is such that $\L^*(\sqrt{\gamma^2 + 2\gamma} -\gamma) =0$.
\end{theorem}
\begin{proof}
We derive the above LDP as a consequence of the G\"{a}rtner-Ellis theorem (see Theorem 2.3.6 in \cite{DZ10}). Therefore, 
we need to estimate the limiting behavior of the normalized cumulant generating function $\L_n (z) = \log \E \left[e^{zU_n}\right]$ . We know that
\begin{align}
\label{LapUn}
\E\left[e^{zU_n}\right]  = \frac{L_n(-X^2 e^{2z})}{L_n(-X^2)}.
\end{align} 
For this propose,  we use the relation between Legendre and Hermite polynomials:
\[L_n (-X^2) = (-1)^n 2^{-2n}(n!)^{-1} H_{2n}(i X)\]
(\cite{Sz75} formula 5.6.1) and the representation by means of the Fourier transform
\begin{align*}
H_n(t) = \frac{e^{t^2}}{\sqrt{2\pi}} \int_{-\infty}^\infty (-is)^n e^{its-\frac{s^2}{4}}ds,
\end{align*}
to get
\begin{align}
\label{firstHer}
L_n(-X^2) = \frac{ (-1)^n 2^{-2n}(n!)^{-1}}{\sqrt{2\pi}}  e^{-X^2} \int_{-\infty}^\infty s^{2n} e^{-sX - \frac{s^2}{4} } ds.
\end{align}
With the change of variables $s \mapsto 2sX$, we obtain
\begin{align*}
L_n(-X^2) = 2\frac{ (-1)^n X^{2n+1}(n!)^{-1}}{\sqrt{2\pi}}  e^{-X^2} \int_{-\infty}^\infty s^{2n} e^{-2sX^2 - X^2s^2 } ds.
\end{align*}
Hence, since $N = 2n$ and $X^2 = 2n\gamma$
\begin{align}
\label{LapUn1}
\E\left[ e^{z U_n}\right] = e^{(N+1)z} e^{-N \gamma( e^{2z} - 1)} \frac{I_N (\gamma e^{2z})}{I_N (\gamma)},
\end{align}
where
\begin{align}
\label{firstHer1}
I_N (\gamma) =  \int_{-\infty}^\infty s^{N} e^{-N\gamma (2s + s^2) } ds\,.\end{align}
To apply the Laplace method, we split the integral in two parts:
\[I_N(\gamma) = I^+_N(\gamma) + I_N^-(\gamma)\, ,\]
with
\[I^\pm_N (\gamma) =  \int_{0}^\infty 
e^{-N g_\pm(\gamma, s)} ds, \quad g_\pm (\gamma, s) = 
-\log s \pm 2 \gamma s + \gamma s^2 
 \,.\]
Let $t_\pm^*(\gamma)$ be the positive solution of 
\begin{align*}
2\gamma t^2 \pm 2\gamma t-1 = 0\,.\end{align*}
The function $g_\pm$ is convex. Its minimum, reached at $t_\pm^*(\gamma)$, is
\[g_\pm (\gamma, t_\pm^*(\gamma)) = -\log t_\pm^*(\gamma) \pm 2\gamma  t_\pm^*(\gamma) + \gamma  t_\pm^*(\gamma)^2\,.\]
A simple look shows that
$g_-(\gamma, t_-^*(\gamma)) < g_+(\gamma, t_+^*(\gamma))$, so that
if we set 
\begin{align}
\label{defpm}
t(z) := t_-^*(\gamma e^{2z}) =\frac{1}{2}\left(1 + \sqrt{1+ \frac{2e^{-2z}}{\gamma}}\right)\,,\end{align}
we get, for every $z \in \R$,
\begin{align}
\label{cvcall}
\frac{1}{2n} \L_n (z) 
\rightarrow
& z- \gamma(e^{2z} -1) - g_- (\gamma e^{2z} , t(z)) + g_- (\gamma , t(0)) =  \L(z)\,.
\end{align}
Let us reparametrize this function in terms of $t \in (1, \infty)$. From (\ref{defpm}) 
\begin{equation}
\label{repar}
z = - \half \log \left(2\gamma t (t-1)\right) \ , \ \frac{dz}{dt} = -\frac{2t-1}{2t(t-1)},\end{equation}
so that, if we denote $t(0) = t_0,$ we can write
\[\L(z) =  \widehat{\L}(t) -  \widehat{\L}(t_0),\]
where \[ \widehat{\L}(t) := - \half \log \left(1 - \frac{1}{t}\right) + \frac{1}{2t}\,.\]

It is clear  that $\widehat{\L}$ (resp. $\L$) is a differentiable function of $t$ (resp. $z$) on its domain $(1, \infty)$ (resp. $\R$).
It is then enough to apply the G\"{a}rtner-Ellis theorem.

Moreover, since $\L'(z) =1/t$
we see that \[\L^* (x) = \sup_{z \in \R} \left(zx - \L (z)\right)= 
\sup_{t \in (1, \infty)}  \left(-\frac{x}{2}\log \left(2\gamma t(t-1)\right) - \widehat{\L}(t) +  \widehat{\L}(t_0)\right)\]
and the supremum is reached in $t = 1/x$ so that 
\[\L^* (x) = x \log x+ \frac{1-x}{2}\log (1-x) -\frac{x}{2}(1 + \log(2\gamma)) + \widehat{\L}(t_0)\,.\]
This function is strictly convex and attains its minimum $0$ at $x= \L'(0) = \frac{1}{t_0}$  and therefore, the proof follows. 
\end{proof}

In the following proposition, we describe the precise asymptotic behavior   of the three first cumulants of $U_n,$ that are latter used to derive the mod-Gaussian convergence.

\begin{proposition}
\label{cvmom}
Let $\kappa^{(3)}_n$ be the third cumulant of $U_n$, then
\begin{align}
\lim_{n \to \infty} \frac{\E[U_n]}{2n} = \frac{1}{t_0} \ &, \  \lim_{n \to \infty} \frac{\Var (U_n)}{2n}  = \sigma^2(\gamma) = \frac{2(t_0-1)}{t_0(2t_0-1)}\\
 \lim_{n \to \infty} \frac{\kappa^{(3)}_n}{2n} & =\kappa(\gamma) =  \frac{4(-2t_0^2 + 5t_0 -1)(t_0-1)}{t_0 (2t_0-1)^3}.
\end{align}
\end{proposition}

\begin{proof}
The structure of the Laplace transform (\ref{LapUn1}) and the representation (\ref{firstHer1}) entail that
\[\lim_{n \to \infty} \frac{1}{2n}\L'_n(z) \rightarrow \L'(z) \ , \  \lim_{n \to \infty} \frac{1}{2n}\L''_ n(z) \rightarrow \L''(z)\ , \   \lim_{n \to infty} \frac{1}{2n}\L^{(3)}_ n(z) \rightarrow \L^{(3)}z)\]
(we leave the details to the reader).  Recall
\footnote{The result $\L'(0) = 1/t_0$  fits with the limiting normalized expectation in Theorem 2.4 of \cite{RSX13}, but the result for $\L''(0)$ does not fit with the limiting normalized variance claimed in the same theorem.}
$\L'(z) =1/t$ and (\ref{repar}), which gives 
\begin{align}
\label{secondderiv}
\L''(z) 
= \frac{2\left(t -1\right)}{t (2t -1)} \  , \  \L''(0) 
= \frac{2\left(t_0 -1\right)}{t_0 (2t_0 -1)}\,.
\end{align}
From (\ref{secondderiv}) we have 
\begin{align*}\L^{(3)} (z) 
= \frac{2\left(-2t^2 + 4t -1\right)}{t^2 (2t-1)^2}\times -\frac{2t(t -1)}{2t -1}= -\frac{4(-2t^2 + 4t -1)(t-1)}{t (2t-1)^3}\,,\end{align*}
and the limiting third cumulant is therefore given by taking $t=t_0$.
\end{proof}
Set $\overline{U}_n=U_n-\E\left[U_n\right]$ and, for $X=\sqrt{2n\gamma}$, denote by $\sigma^2_n (\g):=\frac{\Var(U_n)}{n}.$
\begin{theorem}[Mod-Gaussian convergence]\ \\
As $n \to \infty$, the sequence $\left(\frac{\overline{U}_n }{n^{1/3}}\right)_{n \in \N}$ converges mod-Gaussian with parameters
 $t_n= \sigma_n^2(\g) n^{1/3} $ and the limiting  function $\psi(z)=e^{\kappa(\gamma)\frac{z^3}{3}}.$ The convergence takes place on the whole complex plane.
\end{theorem}

\begin{proof}
Since the support of the measure of orthogonality of  the polynomials $L_n$ is $(0, \infty)$, it is known (Theorem 6.73 in \cite{Sz75}) that 
the zeros of $L_n(x)$ are positive and then
\begin{equation*}
L_n(x)=\prod_{k=1}^n \left(x-\xi_{k,n}\right),
\end{equation*}
with $\xi_{k,n}>0$ for all $k=1,\ldots,n$. The latter, together with equation (\ref{2cl_prob_gen}), implies that the probability generating function $P_n(t)$ of $U_n$  can be written as
\begin{equation}\label{line_pgf}
P_n(t)=\prod_{k=1}^n\frac{X^2t^2+\xi_{k,n}}{X^2+\xi_{k,n}},
\end{equation}
that is, as a product of $n$ polynomials with positive coefficients of degree $2$. To establish the mod-Gaussian convergence, now it suffices to  apply Theorem \ref{newcriterion} and Proposition \ref{cvmom}.
\end{proof}
\begin{remark}
It could be possible to apply Remark \ref{smallo} to this case, but we did not push the computation to that step, not to enlarge this paper.
\end{remark}
\begin{corollary}
\label{cors} With $t_n= \sigma_n^2(\g) n^{1/3} $ and $\psi(z) = e^{\kappa(\gamma) \frac{z^3}{3}}$ ,
\begin{enumerate}
\item
the conclusions of Theorems \ref{ldp} and \ref{clt} hold true, 
\item
the conclusions of Theorems \ref{speed} and \ref{llt} hold true. 
\end{enumerate}
\end{corollary}
\begin{proof}
We have only to prove 2) and actually  it suffices to show that 
the sequence $\left(\frac{\overline{U}_n}{n^{1/3}}\right)_{n \in \N}$ converges mod-Gaussian with index of control $(v,w)=(3,3)$ and a zone of control $[-Dt_n, Dt_n]$, where $D=\frac{1}{8e+96}$. Using (\ref{line_pgf}), $U_n$ can be written as
\begin{equation*}
U_n\stackrel{(d)}{=} 2\sum_{k=1}^nX_{k,n},\qquad X_{k,n}\sim\mathcal{B}\left(\frac{X^2}{X^2+\xi_{k,n}}\right),
\end{equation*}
where $(X_{k,n})_{1\leq k \leq n}$ are independent and $\mathcal{B}(p)$ denotes the Bernoulli distribution with parameter $p.$ 
Since $(X_{k,n})_{1\leq k \leq n}$ are independent, by Theorems 9.1.6 and 9.1.7 in \cite{FMN16}, we have the following bounds on the cumulants
\begin{equation*}
\left|\kappa^{(r)}\left(\overline{U}_n\right)\right|\le n r^{r-2}2^{r-1}2^r,
\end{equation*}
for all $r\ge 2$. Then Lemma 4.2 in \cite{FMN17}, implies that
\begin{equation}\label{bound-l}
\left|\E\left[e^{i\xi \frac{\overline{U}_n}{n^{1/3}}}\right]e^{t_n\frac{\xi^2}{2}}-1\right|\le K|\xi|^3e^{K|\xi|^3},
\end{equation}
for all $\xi\in\left[-Dt_n,Dt_n\right]$ with $D=\frac{1}{(4e+8)8}$ and $K=(2+e)8$.
\end{proof}

\subsection{The fugacity parameter $X=1$}

\begin{theorem}[LDP]\ \\
For $X=1$, the sequence of random variables $\left(\frac{U_n}{2\sqrt n}\right)_{n \in \N}$ satisfies the LDP with rate $\sqrt n$ and good rate function
\begin{equation*}
\Lambda^* (x) = x\log x -x +1 \ , \ (x >0)\,.
\end{equation*}
\end{theorem} 
This result was implicitly obtained in \cite{RSX13} p. 131, with a different proof.
\begin{proof}
 The Laplace transform of $\left(U_n\right)_{n\in \N}$ from Equation (\ref{2cl_mom_gen})
is equal to 
\begin{align*}
\E \left[e^{z U_n}\right] =  \frac{L_n^{-1/2} (-e^{2z})}{L_n^{-1/2} (-1)}.
\end{align*}
We use Perron's formula (see Theorem 8.22.3 in \cite{Sz75} ) to characterize the asymptotic behavior of Laguerre polynomials 
\[L_n^{-1/2} (x) = \frac{1}{2\sqrt \pi} e^{x/2}n^{-1/2} e^{2\sqrt{-nx}} \left(1 + \mathcal O (n^{-1/2})\right)\,,\]
which holds uniformly in any compact set of $\C \setminus [0, \infty)$.
We derive
\begin{equation}
\label{2cl_asm} 
\E \left[e^{z U_n}\right] = e^{\frac{1}{2} (1 -e^{2z}) + 2 \sqrt n (e^z -1)} \left(1 + \mathcal O (n^{-1/2})\right),\end{equation}
uniformly in any compact set of $\{ z | \left| \Im (z)\right| < \frac{\pi}{4} \}$. 
We conclude that
\begin{equation*}
\frac{1}{2\sqrt n} \log \left(\E \left[e^{z U_n}\right] \right)  \rightarrow e^z -1 =: \Lambda(z)\,,
\end{equation*}
and we may apply the G\"artner-Ellis theorem, giving the LDP with rate function given by the Legendre dual of $\Lambda$.
\end{proof}

\begin{remark}
Since $\Lambda^*$ is the Cram\' er transform of the Poisson distribution, we could  expect to have a mod-Poisson convergence (see \cite{FMN16} for the details on mod-Poisson convergence).  We have indeed 
\[\lim_{n \to \infty} \E\left[e^{z U_n}\right] e^{2\sqrt n (e^z - 1)} =  e^{\frac{1}{2} (1 -e^{2z})}\]
but this convergence is valid only for $\{ z | \left| \Im (z)\right| < \frac{\pi}{4} \}.$ Therefore, the assumptions required for the mod-Poisson convergence are not satisfied.
\end{remark}
In contrast, the renormalized sequence $\left(\frac{U_n}{n^{1/6}}\right)_{n \in \N}$ converges mod-Gaussian. Indeed, from (\ref{2cl_asm}), the cumulant generating series of $\left(\frac{U_n}{n^{1/6}}\right)_{n \in \N}$  can be written as
\begin{align*}
\log \left(\E \left[e^{\frac{z U_n}{n^{1/6}}}\right] \right) &= \frac{1}{2} \left(1 -e^{\frac{2z}{n^{1/6}}}\right) + 2 \sqrt n \left(e^{\frac{z}{n^{1/6}}} -1\right)+o(1)\\
&=2z n^{1/3} + z^2 n^{1/6} + \frac{z^3}{3} + o(1),
\end{align*}
for all $\lbrace z | \left| \Im (z)\right| < \frac{ \pi}{4} n^{1/6} \rbrace$, which readily implies the following theorem.
\begin{theorem}[Mod-Gaussian convergence]\ \\
As $n \to \infty$, the sequence $\left(\frac{U_n-2\sqrt{n}}{ n^{1/6}}\right)_{n \in \N}$ converges mod-Gaussian with parameters $t_n=2 n^{1/6}$ and the limiting  function $\psi(z)=e^{\frac{z^3}{3}}.$ The convergence takes place on the whole complex plane.
\end{theorem}

\begin{corollary}
With $t_n = 2n^{1/6}$ and $\psi(x) = e^{\frac{x^3}{3}}$, the conclusions of Theorems \ref{ldp} and \ref{clt} hold true.
\end{corollary}

\section{The limiting results for the two charge ensemble on the circle} \label{results: circle}
Let us recall that $V_n \in [0,2n]$ for every $n$.
\begin{theorem}[LDP]\ \\
The sequence of random variables $\left(\frac{V_n}{n}\right)_{n \in \N}$ satisfies the LDP in $(0,2)$ with speed $n$ 
 and good rate function $\Lambda^*$  which is  the Legendre dual of 
\begin{align}
\nonumber
\Lambda(z) := 
2\rho e^{z}  \arctan \frac{1}{\rho e^{z }} - 2\rho \arctan \frac{1}{\rho} + \log \frac{\rho^2 e^{2z} + 1}{\rho^2 + 1}\,.
\end{align}
The function $x \mapsto \Lambda^*(x)$ is convex on $(0,2)$, reaches its unique minimum $0$ at $\Lambda'(0) = 2\rho \arctan 1/\rho$ and satisfies
\begin{align} \Lambda^*(0) = -2\rho \arctan \left(\frac{1}{\rho}\right) -\log (1+\rho^2) \ , \ \lim_{x \downarrow 0}  (\Lambda^*)'(x) = - \infty\ , \ 
\lim_{x \uparrow 2} \Lambda^*(x) = \infty\,.
\end{align}

\end{theorem}
\bigskip

It means that, when $n$ tends to infinity,  $V_n/ n$ tends to $\Lambda'(0)$ in probability and that
\[\lim_{n \to \infty} n^{-1}\log \mathbb P(V_n = 0) = - \Lambda^*(0), \quad \lim_{n \to \infty} n^{-1}\log \mathbb P(V_n = 2n) = - \infty\,.\]

\begin{proof}
We use the G\"{a}rtner-Ellis theorem. From (\ref{2cc_mom_gen}), the moment generating function of the sequence $(V_n)_{n \in \N}$ can be written as
\begin{align*}
 \E \left[ e^{zV_n}\right] = \prod_{k=1}^n \frac{ \rho^2 e^{2z}  + \left(\frac{2k-1}{2n}\right)^2  }{  \rho ^2 + \left(\frac{2k-1}{2n}\right)^2  }\,.
\end{align*} 
From the convergence of the Riemann sums we have 
\begin{align}
\label{Riem}
\lim_{n \to \infty} \frac{1}{n} \log E\left[ e^{z V_n} \right]&=  \int_0^{1} \log \left(\frac{\rho^2 e^{2z} + t^2}{\rho^2 + t^2}\right) dt\\
\nonumber & =
2\rho e^{z}  \arctan \frac{1}{\rho e^{z }} - 2\rho \arctan \frac{1}{\rho} + \log \frac{\rho^2 e^{2z} + 1}{\rho^2 + 1} =: \Lambda(z)\,.
\end{align}
The first derivative is clearly
\begin{align}\label{derivlam}\Lambda'(z) &= \int_0^1 \frac{2\rho^2e^{2z}}{\rho^2e^{2z} + t^2} dt = 2\rho e^z \arctan\frac{1}{\rho e^z}\,,\end{align}
and all the remaining assertions are straightforward.
\end{proof}

In the following proposition, we limiting results on the three first cumulants of $V_n$.
\begin{proposition}
If $\kappa^{(3)}_n$ is the third cumulant of $V_n$, then
\begin{align}
\lim_{n \to \infty} \frac{\E\left[V_n\right]}{n} = 2 \rho \arctan \frac{1}{\rho} \ &, \ \lim_{n \to \infty} \frac{\Var(V_n)}{n} =  2\rho \arctan\frac{1}{\rho} -  \frac{2\rho^2 }{1+\rho^2} =: \sigma^2(\rho)\\
\lim_{n \to \infty} \frac{\kappa^{(3)}_n}{n} &=   \kappa(\rho) = 2\rho \arctan \frac{1}{\rho} -2\rho^2\frac{3+\rho^2}{(1+\rho^2)^2}
\end{align}
\end{proposition}

\begin{proof}
The proof is the same as in Section 2, with $\Lambda'$ given in (\ref{derivlam}), so that 
\begin{align*}
\Lambda'' (z) &= 2\rho e^z \arctan \frac{1}{\rho e^z} -\frac{2\rho^2 e^{2z}}{1+\rho^2 e^{2z}}\,,\\
\Lambda^{(3)} (z) &= 2\rho e^z \arctan \frac{1}{\rho e^z} -2\rho^2e^{2z}\frac{3+\rho^2e^{2z}}{(1+\rho^2e^{2z})^2}\,.
\end{align*}
\end{proof}
These limits coincide with the results  of Theorem 2.2 in \cite{SS14}, up to a change of notation since $L_N (Nr) = V_n$ with $N=2n$ and $\rho = 2r.$\
\\
We set $\overline{V}_n = V_n - \E[V_n]$ and denote by $\sigma_n^2(\rho) := \frac{\Var(V_n)}{n}.$
\begin{theorem}[Mod-Gaussian convergence]\ \\ \label{mod2cc}
As $n\to \infty,$ the sequence $\left(\frac{\overline{V}_n}{n^{1/3}} \right)_{n \in \N}$ converges mod-Gaussian with parameters $t_n :=  n^{1/3} \sigma_n^2(\rho)$ and the limiting function
 $\psi(z):=e^{\kappa(\rho)\frac{z^3}{6}} \,.$ The convergence takes place on the whole complex plane.
\end{theorem}
\proof

From (\ref{2cc_prob_gen}), we see that the probability generating function of $V_n$ is a product of $n$ positive polynomials of degree $2$.  
It is then enough to apply Proposition \ref{newcriterion} and the above proposition. 

\begin{corollary}
With $t_n = n^{1/3}\sigma_n^2(\rho)$ and $\psi(x) = e^{\kappa(\rho)\frac{x^3}{6}}$, the conclusions of Theorems  \ref{ldp}, \ref{clt}, \ref{speed}, \ref{llt} hold true.
\end{corollary}

The proof is the same as that of Corollary \ref{cors},  
since from (\ref{2cc_mom_gen}) $V_n$ can be written as
\begin{equation}
V_n\stackrel{(d)}{=} 2\sum_{k=1}^n X_{k,n},\qquad X_{k,n}\sim\mathcal{B}\left(\frac{(2n\rho)^2}{(2n\rho)^2+(2k-1)^2}\right),
\end{equation}
where $(X_{k,n})_{1\leq k \leq n}$ are independent. 

\bibliographystyle{plain}
\bibliography{Bib_mod_phi}

\begin{thebibliography}{10}

\bibitem{Be73}
E.A. Bender.
\newblock Central and local limit theorems applied to asymptotic enumeration.
\newblock {\em Journal of Combinatorial Theory, Series A}, 15(1):91--111, 1973.

\bibitem{BS09}
A.~Borodin and C.~D. Sinclair.
\newblock The {G}inibre ensemble of real random matrices and its scaling
  limits.
\newblock {\em Comm. Math. Phys.}, 291(1):177--224, 2009.

\bibitem{BMN17}
M.~Dal~Borgo, P.~L. M{\'e}liot, and A.~Nikeghbali.
\newblock Local limit theorems and mod-$\phi$ convergence.
\newblock In preparation, 2017.

\bibitem{DKN15}
F.~Delbaen, E.~Kowalski, and A.~Nikeghbali.
\newblock Mod-{$\phi$} convergence.
\newblock {\em Int. Math. Res. Not. IMRN}, (11):3445--3485, 2015.

\bibitem{DZ10}
A.~Dembo and O.~Zeitouni.
\newblock {\em Large deviations techniques and applications}, volume~38 of {\em
  Stochastic Modelling and Applied Probability}.
\newblock Springer-Verlag, Berlin, 2010.
\newblock Corrected reprint of the second (1998) edition.

\bibitem{FMN16}
V.~F{\'e}ray, P.~L. M{\'e}liot, and A.~Nikeghbali.
\newblock {\em Mod-$\phi$ Convergence}.
\newblock SpringerBriefs in Probability and Mathematical Statistics. Springer
  International Publishing, 1 edition, 2016.

\bibitem{FMN17}
V.~F{\'e}ray, P.~L. M{\'e}liot, and A.~Nikeghbali.
\newblock Mod-$\phi$ convergence, ii: Estimates of the speed of convergence.
\newblock Preprints Submitted, 2017.

\bibitem{F84B}
P.~J. Forrester.
\newblock Analogues between a quantum many body problem and the log-gas.
\newblock {\em J. Phys. A}, 17(10):2059--2067, 1984.

\bibitem{F84A}
P.~J. Forrester.
\newblock An exactly solvable two component classical {C}oulomb system.
\newblock {\em J. Austral. Math. Soc. Ser. B}, 26(2):119--128, 1984.

\bibitem{F10}
P.~J. Forrester.
\newblock {\em Log-gases and random matrices}, volume~34 of {\em London
  Mathematical Society Monographs Series}.
\newblock Princeton University Press, Princeton, NJ, 2010.

\bibitem{FN07}
P.~J. Forrester and T.~Nagao.
\newblock Eigenvalue statistics of the real ginibre ensemble.
\newblock {\em Phys. Rev. Lett.}, 99:050603, Aug 2007.

\bibitem{MPT15}
L.C. Garcia~del Molino, K.~Pakdaman, and J.~Touboul.
\newblock The heterogeneous gas with singular interaction: Generalized circular
  law and heterogeneous renormalized energy.
\newblock {\em J. Phys. A}, 48(4):045208, 2015.

\bibitem{G65}
J.~Ginibre.
\newblock Statistical ensembles of complex, quaternion, and real matrices.
\newblock {\em J. Mathematical Phys.}, 6:440--449, 1965.

\bibitem{JKN11}
J.~Jacod, E.~Kowalski, and A.~Nikeghbali.
\newblock Mod-{G}aussian convergence: new limit theorems in probability and
  number theory.
\newblock {\em Forum Math.}, 23(4):835--873, 2011.

\bibitem{KN12}
E.~Kowalski and A.~Nikeghbali.
\newblock Mod-{G}aussian convergence and the value distribution of
  {$\zeta(\frac12+it)$} and related quantities.
\newblock {\em J. Lond. Math. Soc. (2)}, 86(1):291--319, 2012.

\bibitem{RSX13}
B.~Rider, C.~D. Sinclair, and Y.~Xu.
\newblock A solvable mixed charge ensemble on the line: global results.
\newblock {\em Probab. Theory Related Fields}, 155(1-2):127--164, 2013.

\bibitem{SS14}
C.~Shum and C.D. Sinclair.
\newblock {A Solvable Two-Charge Ensemble on the Circle}.
\newblock {\em ArXiv e-prints}, April 2014.

\bibitem{Sz75}
G.~Szeg{\H{o}}.
\newblock Orthogonal polynomials.
\newblock In {\em Amer. Math. Soc. Colloq. Publ}, volume~23, 1975.

\end{thebibliography}

\end{document}